\documentclass[11pt,oneside,english,jou]{amsart}
\usepackage[T1]{fontenc}
\usepackage[latin9]{inputenc}
\usepackage{verbatim}
\usepackage{amsthm}
\usepackage{amssymb}

\makeatletter
\numberwithin{equation}{section}
\numberwithin{figure}{section}
  \theoremstyle{plain}
  \newtheorem*{thm*}{Theorem}
\theoremstyle{plain}
\newtheorem{thm}{Theorem}[subsection]
  \theoremstyle{definition}
  \newtheorem{defn}[thm]{Definition}
  \theoremstyle{plain}
  \newtheorem{prop}[thm]{Proposition}
  \theoremstyle{plain}
  \newtheorem{lem}[thm]{Lemma}
 \theoremstyle{definition}
  \newtheorem{example}[thm]{Example}
  \theoremstyle{remark}
  \newtheorem{rem}[thm]{Remark}
  \theoremstyle{remark}
  \newtheorem{claim}[thm]{Claim}
  \theoremstyle{plain}
  \newtheorem{question}[thm]{Question}

\newcommand{\N}{\mathbb{N}}

\makeatother

\usepackage{babel}

\begin{document}

\title{On Relative Extreme Amenability}

\subjclass[2010]{Primary: 54H20. Secondary: 05D10, 22F05, 37B05}
\keywords{Topological groups actions, extreme amenability, universal minimal space, Kechris-Pestov-Todorcevic correspondance, Fra\"{i}ss\'e theory}
\date{October, 2013}

\author{Yonatan Gutman \& Lionel Nguyen Van Th\'e}
\begin{abstract} The purpose of this paper is to study the notion of \textit{relative extreme amenability} for pairs of topological groups. We give a characterization by a fixed point
property on universal spaces. In addition we introduce the concepts
of an \textit{extremely amenable interpolant} as well as \textit{maximally
relatively extremely amenable} pairs and give examples. It is shown
that relative extreme amenability does not imply the existence of
an extremely amenable interpolant. The theory is applied to generalize
results of \cite{KPT05} relating to the application of Fra\"{i}ss\'e theory
to theory of Dynamical Systems. In particular, new conditions enabling
to characterize universal minimal spaces of automorphism groups of
Fra\"{i}ss\'e structures are given.
\end{abstract}
\maketitle

\section{Introduction}

The goal of this paper is to study the notion of \textit{relative
extreme amenability}: a pair
of topological groups $H\subset G$ is called relatively extremely
amenable if whenever $G$ acts continuously on a compact space, there is an $H$-fixed
point. This notion was isolated by the second author while investigating transfer properties between Fra\"{i}ss\'e theory and dynamical systems along the lines of \cite{KPT05}, and the corresponding results appears in \cite{NVT-}. 
We now provide a short description of the contents of the present article
and some of the results. Section \ref{sec:Preliminaries} contains
notation. Subsection \ref{sub:Universal-spaces.} recalls the notion
of \textit{universal spaces. }In subsection\textit{ \ref{sub:A-Characterization-Rea}}
it is shown that $(G,H)$ is relatively extremely amenable if and
only if there exists a universal $G$-space with a $H$-fixed point.
In subsection \ref{sub:Extremely-Amenable-Interpolants} the notion
of \textit{extremely amenable interpolant} is introduced and an example
of a non trivial interpolant is given. Subsection \ref{sub:Order-fixing-groups}
contains technical lemmas. In subsection \textit{\ref{sub:Maximally-Rea}
}the notions of \textit{maximal relative extreme amenability }and
\textit{maximal extreme amenability} are introduced and illustrated.
It is also shown that relative extreme amenability does not imply
the existence of an extremely amenable interpolant and that $Aut(\mathbb{Q},<)$
is maximally extremely amenable in $S_{\infty}$. Subsections \ref{sub:Fraisse-Theory}
and \ref{sub:weak-ordering-Property} deal with applications to a
beautiful theory developed in \cite{KPT05} - the application of Fra\"{i}ss\'e
theory to the theory of Dynamical Systems. In subsection \ref{sub:Fraisse-Theory}
the following theorem is shown (see subsection for the definitions
of the various terms appearing in the statement):
\begin{thm*}
Let $\{<\}\subset L$,$L_{0}=L\setminus\{<\}$ be signatures, $K_{0}$
a Fra\"{i}ss\'e class in $L_{0}$, $K$ an order Fra\"{i}ss\'e expansion of $K$ in $L$, $F_{0}=Flim(K_{0})$,
$F=Flim(K)$. Let $G_{0}=Aut(F_{0})$ and $G=Aut(F)$.
Denote $<^{F}=<_{0}$ and $X_{K}=\overline{G_{0}<_{0}}$\textup{.
$(G_{0},G)$ is }relatively extremely amenable and $Fix_{X_{K}}(G)$
is transitive w.r.t $X_{K}$\textup{ if and only if $X_{K}$ is the
universal minimal space of $G_{0}$}.
\end{thm*}
In subsection \ref{sub:weak-ordering-Property} the \textit{weak ordering
property} is introduced and it is proven that if $(G_{0},G)$ is relatively
extremely amenable then the weak ordering property implies the ordering
property. Finally in subsection \ref{sub:Conjecture.} a question is formulated.

\

\textbf{Ackowledgements}: This project began while both of us were attending the thematic program on Asymptotic and Geometric Analysis taking place in Fall 2010 at the Fields Institute in Toronto. We would therefore like to ackowledge the support of the Fields Institute, and thank the organizers Vitali Milman, Vladimir Pestov and Nicole Tomczak-Jaegermann for having made this work possible. We would also like to thank Todor Tsankov for mentioning Peter Cameron's article \cite{Cam76} regarding Theorem \ref{thm:AutQ}.

\section{\label{sec:Preliminaries}Preliminaries}

We denote by $(G,X)$ a topological dynamical system (t.d.s), where
$G$ is a (Hausdorff) topological group and $X$ is a compact (Hausdorff)
topological space. We may also refer to $X$ as a $G$-space. If it
is desired to distinguish a specific point $x_{0}\in X$, we write
$(G,X,x_{0}).$ Given a continuous action $(G,X)$ and $x\in X$,
denote by $Stab{}_{G}(x)=Stab(x)=\{g\in G\,|\, gx=x\}\subset G$,
the subgroup of elements of $G$ fixing $x$, and for $H\subset G$
denote by $Fix_{X}(H)=Fix(H)=\{x\in X\,|\,\forall h\in H\, hx=x\}\subset X$,
the set of elements of $X$, fixed by $H$. Note that $Fix_{X}(H)$
is a closed set. Given a linear order $<$ on a set $D$, we denote
by $<^{*}$ the linear ordering defined on $D$ by $a<^{*}b\Leftrightarrow b<a$ for all
$a,b\in D.$

\section{Results}

\subsection{\label{sub:Universal-spaces.}Universal spaces.}

Let $G$ be a topological group. The topological dynamical system
(t.d.s.) $(G,X)$ is said to be \textbf{minimal} if $X$ and $\emptyset$
are the only $G$-invariant closed subsets of $X$. By Zorn's lemma
each $G$-space contains a minimal $G$-subspace. $(G,X)$ is said
to be \textbf{universal} if any minimal $G$-space $Y$ is a $G$-factor
of $X$. One can show there exists a minimal and universal $G$-space
$U_{G}$ unique up to isomorphism. $(G,U_{G})$ is called the \textbf{universal
minimal space of $G$} (for existence and uniqueness see for example \cite{Usp01} 
, or the more recent \cite{GL11}).
$(G,X,x_{0})$ is said to be \textbf{transitive} if $\overline{Gx_{0}}=X$.
One can show there exists a transitive t.d.s $(G,A_{G},a_{0})$, unique
up to isomorphism, such that for any transitive t.d.s $(G,Y,y_{0})$,
there exists a $G$-equivariant mapping $\phi_{Y}:(G,A_{G},a_{0})\rightarrow(G,Y,y_{0})$
such that $\phi(a_{0})=y_{0}$. $(G,A_{G},a_{0})$ is called the \textbf{greatest
ambit}. Because any minimal subspace of $A_{G}$ is isomorphic to
the universal minimal space, $A_{G}$ is universal. Note that if $A_{G}$
is not minimal (e.g., this is the case if $A_{G}$ is not distal see
\cite{dV93} IV(4.35)), then it is an example of a \textit{non-minimal}
universal space.

\subsection{\label{sub:A-Characterization-Rea}A Characterization of Relative
Extreme Amenability}

Recall the following classical definition (originating in \cite{M66}):
\begin{defn}
Let $G$ be a topological group. $G$ is called\textbf{ extremely
amenable} if any t.d.s $(G,X)$ has a $G$-fixed point, i.e. there
exists $x_{0}\in X$, such that for every $g\in G$, $gx_{0}=x_{0}$.
\end{defn}
It is easy to see that for $G$ to be extremely amenable is equivalent
to $U_{G}=\{\ast\}$. Here is a generalization of the previous definition
which appears in \cite{NVT-}:
\begin{defn}
Let $G$ be a topological group and $H\subset G$, a subgroup. The
pair $(G,H)$ is called \textbf{relatively extremely amenable }if
any t.d.s $(G,X)$ has a $H$-fixed point, i.e. there exists $x_{0}\in X$,
such that for every $h\in H$, $hx_{0}=x_{0}$. \end{defn}
\begin{prop}
\label{thm:rea=00003D00003DFix on Universal}Let $G$ be a topological
group and $H\subset G$, a subgroup. The following conditions are
equivalent:\end{prop}
\begin{enumerate}
\item \label{enu:(G,H) rea}The pair $(G,H)$ is relatively extremely amenable.
\item \label{enu:U_G fixed H point}$U_{G}$ has a $H$-fixed point.
\item \label{enu:Universal T with H subset Fix(t)}There exists a universal
$G$-space $T_{G}$ and $t_{0}\in T_{G}$ which is $H$-fixed. \end{enumerate}
\begin{proof}

(\ref{enu:(G,H) rea})$\Rightarrow$(\ref{enu:U_G fixed H point}).
If $(G,H)$ is relatively extremely amenable, then by definition $(G,U_{G})$
has a $H$-fixed point.

(\ref{enu:U_G fixed H point})$\Rightarrow$(\ref{enu:Universal T with H subset Fix(t)}).
Trivial.

(\ref{enu:Universal T with H subset Fix(t)})$\Rightarrow$(\ref{enu:(G,H) rea}).
Let $X$ be a minimal $G$-space. By universality of $T_{G}$, there
exists a surjective $G$-equivariant mapping $\phi:(G,T_{G})\rightarrow(G,X)$.
Denote $x=\phi(t_{0}).$ Clearly for every $h\in H$, $hx=h\phi(t_{0})=\phi(ht_{0})=\phi(t_{0})=x$
\end{proof}
It is well-known that a non-compact locally compact group cannot be
extremely amenable. Here is a strengthening of this fact:
\begin{prop}
Let $G$ be a non-compact locally compact group and $\{e\}\subsetneq H\subset G$,
a subgroup. The pair $(G,H)$ is not relatively extremely amenable.\end{prop}
\begin{proof}
By Veech's Theorem (\cite{V77}) $G$ acts freely on $U_{G}$. Now
use Proposition \ref{thm:rea=00003D00003DFix on Universal}(\ref{enu:U_G fixed H point}).
\end{proof}

\subsection{\label{sub:Extremely-Amenable-Interpolants}Extremely Amenable Interpolants}
\begin{defn}
Let $G$ be a topological group and $H\subset G$, a subgroup. An
extremely amenable group $E$ is called an \textbf{extremely amenable
interpolant }for the pair \textbf{$(G,H)$} if $H\subset E\subset G$.
\end{defn}
The following lemma is trivial:
\begin{lem}
Let $G$ be a topological group and $H\subset G$, a subgroup. If
there exists an extremely amenable interpolant for the pair $(G,H)$,
then $(G,H)$ is relatively extremely amenable.
\end{lem}
Here is an example of a non trivial extremely amenable interpolant
$E$ for a pair $(G,H)$, in the sense that neither $E=G$, nor $E=H$:
\begin{example}
Let $Q$ be the Hilbert cube. Recall that by a result of Uspenskij
(Theorem 9.18 of \cite{K95}), $Homeo(Q),$ equipped with the compact-open
topology, is a universal Polish group, in the sense that any Polish
group embeds inside it through a homomorphism. Let $Homeo_{+}(I)$
be the group of increasing homeomorphisms of the interval $I$, equipped
with the compact-open topology. By a result of Pestov (see \cite{P98})
$Homeo_{+}(I)$ is extremely amenable. Let $\phi:Homeo_{+}(I)\hookrightarrow Homeo(Q)$
be an embedding through a homomorphism. Let $f:I\rightarrow I$ given
by $f(x)=x^{2}$. Notice $f\in Homeo_{+}(I)$. Denote $G=Homeo(Q),$
$E=\phi(Homeo_{+}(I))$ and $H=\phi(\{f^{n}\,|\, n\in\mathbb{Z}\})$.
Notice $H\subsetneq E\subsetneq G$. $E$ is clearly an extremely
amenable interpolant for $(G,H)$, but $G$ (which acts homogeneously
on $Q$) and $H$ (which is isomorphic to $\mathbb{Z}$) are not extremely
amenable.

A natural question is if any relatively extremely amenable pair has
an extremely amenable interpolant. Theorem \ref{thm:rea_not_iea}
in the next subsection answers the question in the negative.
\end{example}

\subsection{\label{sub:Order-fixing-groups}Order fixing groups}

Let $S_{\infty}$ be the permutation group of the integers $\mathbb{Z}$,
equipped with the pointwise convergence topology. Let $F$ be an infinite
countable set and fix a bijection $F\simeq\mathbb{Z}$. Let $LO(F)\subset\{0,1\}^{F\times F}$,
be the space of linear orderings on $F$, equipped with the pointwise
convergence topology. Under the above mentioned bijection $LO(F)$
becomes an $S_{\infty}$-space. By Theorem 8.1 of \cite{KPT05} $U_{S_{\infty}}=LO(F)$.
Notice that we consider $F$ as a set and not a topological space.
In this subsection we will use $F=\mathbb{Z}$ and $F=\mathbb{Q}$,
considered as infinitely countable sets with convenient enumerations
(bijections) and the corresponding dynamical systems $(S_{\infty},LO(\mathbb{Z}))$
and $(S_{\infty},LO(\mathbb{Q}))$.
\begin{lem}
\label{lem:Fix(Aut(Z))}Let $<\in LO(\mathbb{Z})$ be the usual linear
order on $\mathbb{Z}$, i.e. the order for which $n<n+1$ for every $n \in \mathbb{Z}$. Then
\begin{enumerate}
\item \label{enu:Stab_Z_<}$Stab_{\mathbb{Z}}(<)=\{T_{a}|\, a\in\mathbb{Z}\}$,
where $T_{a}:\mathbb{Z}\rightarrow\mathbb{Z}$ is given by $T_{a}(x)=x+a$.
\item \label{enu:fix_LO(Z)}$Fix_{LO(\mathbb{Z})}(Stab_{\mathbb{Z}}(<))=\{<,<^{*}\}.$
\end{enumerate}
\end{lem}

\begin{proof}

(\ref{enu:Stab_Z_<}) Let $T\in Stab(<)$. Denote $a=T(0).$ Notice that for all $x>1$,
$T(x)>T(1)>a$ and for all $x<0$, $T(x)<a$. As $T$ is onto we must
have $T(1)=a+1$. Similarly for all $x\in\mathbb{Z}$, $T(x)=x+a$,
which implies $T=T_{a}.$

(\ref{enu:fix_LO(Z)}) Let $\prec\in Fix_{LO(\mathbb{Z})}(Stab(<))$. We claim that $\prec=<$
or $\prec=<^{*}$. Indeed $0\prec1$ or $1\prec0$. In the first case
applying $T_{a}\in Stab(<),$ we have for all $a\in\mathbb{Z}$, $a\prec a+1$.
This implies $\prec=<$. Similarly in the second case for all $a\in\mathbb{Z}$,
$a+1\prec a$ which implies$\prec=<^{*}$.
\end{proof}
Let $<\in LO(\mathbb{Q})$ be the usual order on $\mathbb{Q}$. In
the following lemma, we follow the standard convention and write $Aut(\mathbb{Q},<)$ instead of $Stab_{S_{\infty}}(<)\subset S_{\infty}$.
\begin{lem}
\label{lem:Fix(Aut(Q))}Let $<\in LO(\mathbb{Q})$ be the usual linear
order on $\mathbb{Q}$, then $$Fix_{LO(\mathbb{Q})}(Aut(\mathbb{Q},<~))=\{<,<^{*}\}.$$ \end{lem}
\begin{proof}
Let $\prec\in Fix_{LO(\mathbb{Q})}(Aut(\mathbb{Q},<))$. Note that $0\prec1$
or $1\prec0$. In the first case, let $q_{1},q_{2}\in\mathbb{Q}$
with $q_{1}<q_{2}$ and define $T:\mathbb{Q}\rightarrow\mathbb{Q}$
with $Tx=(q_{2}-q_{1})x+q_{1}$. Note that $T\in Aut(\mathbb{Q},<)$. Hence, $q_{1}=T(0)\prec T(1)=q_{2}$. As the argument works for any $q'_{1}<q'_{2}$
we have $\prec=<$. The second case is similar and implies $\prec=<^{*}.$
\end{proof}

\subsection{Maximally Relatively Extremely Amenable Pairs}

\label{sub:Maximally-Rea}
\begin{prop}
\label{thm:mrea}Let $G$ be a topological group, then there exists
a subgroup $H\subset G$, such that $(G,H)$ is relatively extremely
amenable and there exists no subgroup $H\subset H'\subset G$, such
that $(G,H')$ is relatively extremely amenable. \end{prop}
\begin{proof}
By Zorn's lemma it is enough to show that any chain w.r.t. inclusion
$\{G_{\alpha}\}_{\alpha\in A}$ such that $(G,G_{\alpha})$ is relatively
extremely amenable, has a maximal element. Note that if $G_{\alpha}\subset G_{\alpha'}$,
then $Fix_{U_{G}}(G_{\alpha'})\subset Fix_{U_{G}}(G_{\alpha})$. In
particular for any finite collection $\alpha_{1},\alpha_{2},\ldots,\alpha_{n}\in A$,
we have $\bigcap_{i=1}^{n}Fix_{U_{G}}(G_{\alpha_{i}})\neq\emptyset$,
which implies by a standard compactness argument $\bigcap_{\alpha\in A}Fix_{U_{G}}(G_{\alpha})\neq\emptyset$.
This in turn implies that $Fix_{U_{G}}(\bigcup_{\alpha\in A}G_{\alpha})\neq\emptyset$,
which finally implies $(G,\bigcup_{\alpha\in A}G_{\alpha})$ is relatively
extremely amenable by Proposition \ref{thm:rea=00003D00003DFix on Universal}(\ref{enu:U_G fixed H point}).\end{proof}
\begin{defn}
A pair $(G,H)$ as in Proposition \ref{thm:mrea} is called \textbf{maximally
relatively extremely amenable}.
\end{defn}
Similarly to the previous theorem and definition we have:
\begin{prop}
\label{thm:mea}Let $G$ be a topological group, then there exists
a subgroup $H\subset G$, such that $H$ is extremely amenable and
there exists no subgroup $H\subset H'\subset G$, such that $H'$
is extremely amenable. \end{prop}
\begin{proof}
By Zorn's lemma it is enough to show that any chain w.r.t. inclusion
$\{G_{\alpha}\}_{\alpha\in A}$ such that $G_{\alpha}\subset G$ and
$G_{\alpha}$ is extremely amenable, has a maximal element. Let $(\bigcup_{\alpha\in A}G_{\alpha},X)$
be a dynamical system. By assumption for any $\alpha\in A$, $Fix_{X}(G_{\alpha})\neq\emptyset$.
In addition if $G_{\alpha}\subset G_{\alpha'}$, then $Fix_{X}(G_{\alpha'})\subset Fix_{X}(G_{\alpha})$.
We now continue as in the proof of Theorem \ref{thm:mrea} to conclude
$\bigcup_{\alpha\in A}G_{\alpha}$ is extremely amenable.\end{proof}
\begin{defn}
A subgroup $H\subset G$ as in Proposition \ref{thm:mea} is called
\textbf{maximally extremely amenable in $G$}.\end{defn}
\begin{rem}
It was pointed out in \cite{P02} that if $H$ is second countable
(Hausdorff) group then there always exists an extremely amenable group
$G$ such that $H\subset G$. Indeed by \cite{U90} $H\subset Iso(\mathbb{U})$
the group of isometries of Urysohn's universal complete separable
metric space $\mathbb{U}$, equipped with the compact-open topology,
and by \cite{P02}, $Iso(\mathbb{U})$ is extremely amenable.\end{rem}
\begin{thm}
\label{thm:S_infinity,Stab(<) is mrea}Let $G=S_{\infty}$ be the
permutation group of the integers, equipped with the pointwise convergence
topology. Let $<$ be the usual order on $\mathbb{Z}$ and $H=Stab_{\mathbb{Z}}(<)\subset G$.
The pair $(G,H)$ is maximally relatively extremely amenable.\end{thm}
\begin{proof}
By Theorem 8.1 of \cite{KPT05} $U_{G}=LO(\mathbb{Z})$, the space
of linear orderings on $\mathbb{\mathbb{Z}}$. By Proposition \ref{thm:rea=00003D00003DFix on Universal}(\ref{enu:U_G fixed H point})
$(G,H)$ is relatively extremely amenable. Assume that there exists
a subgroup $E$, with $H\subset E\subset G$ such that $(G,E)$ is
a relatively extremely amenable. Evoking again Proposition \ref{thm:rea=00003D00003DFix on Universal}(\ref{enu:U_G fixed H point}),
there exists $\prec\in U_{G}$, so that $E\subset Stab(\prec)$. As
$H\subset E\subset Stab(\prec)$, conclude by Lemma \ref{lem:Fix(Aut(Z))}(\ref{enu:fix_LO(Z)})
that $\prec\in\{<,<^{*}\}$. As $H=Stab(<)=Stab(<^{*})$, we conclude
in both cases $E=H$. \end{proof}
\begin{lem}
\label{lem:mrea+nea->niea}If $(G,H)$ is maximally relatively extremely
amenable and neither $G$ nor $H$ are extremely amenable, then $(G,H)$
does not admit an extremely amenable interpolant. \end{lem}
\begin{proof}
Assume for a contradiction that there exists an extremely amenable subgroup
$E$, with $H\subset E\subset G$. Notice that $(G,E)$ is relatively extremely
amenable which constitutes a contradiction with the fact that $(G,H)$
is maximally relatively extremely amenable. \end{proof}
\begin{thm}
\label{thm:rea_not_iea}There exists a relatively extremely amenable
pair $(G,H)$ which which does not admit an extremely amenable interpolant.\end{thm}
\begin{proof}
Let $G=S_{\infty}$ be the permutation group of the integers, equipped
with the pointwise convergence topology. Let $<$ be the usual order
on $\mathbb{Z}$ and $H=Stab(<)\subset G$. By Theorem \ref{thm:S_infinity,Stab(<) is mrea}
$(G,H)$ is maximally relatively extremely amenable. Clearly $G$
is not extremely amenable as $U_{G}\neq\{\ast\}$. By Lemma \ref{lem:Fix(Aut(Z))}(\ref{enu:Stab_Z_<})
$H=\{T_{a}|\, a\in\mathbb{Z}\}\cong\mathbb{Z}$, where the second
equivalence is as topological groups. This implies $H$ is not extremely
amenable. Now invoke Lemma \ref{lem:mrea+nea->niea}. \end{proof}
\begin{thm}\label{thm:AutQ}
$Aut(\mathbb{Q},<)$ is maximally extremely amenable in $S_{\infty}$.\end{thm}
\begin{proof}
By \cite{P98} $Aut(\mathbb{Q},<)$ is extremely amenable. Now we
can proceed as in the proof of Theorem \ref{thm:rea_not_iea} using
Lemma \ref{lem:Fix(Aut(Q))}.
\end{proof}

\begin{rem} Even though the previous result never appeared in print, Todor Tsankov pointed out that it can be derived from an earlier result by Cameron. Indeed, the article \cite{Cam76} 
allows a complete description of the closed subgroups $G$ of $S_{\infty}$ containing $Aut(\mathbb{Q})$ (essentially, there are only five of them, see \cite{BP-} for an explicit description) 
and it can be verified that among those, only $Aut(\mathbb{Q})$ is extremely amenable. 
\end{rem}

\subsection{Applications in Fra\"{i}ss\'e Theory}

\label{sub:Fraisse-Theory}

The following two sections deal with applications Fra\"{i}ss\'e Theory.
Two general references for this theory are \cite{F00} and \cite{H93}.
We follow the exposition and notation of \cite{KPT05}. 

Let $\{<\}\subset L$,$L_{0}=L\setminus\{<\}$
be signatures, $K_{0}$ a Fra\"{i}ss\'e class in $L_{0}$, $K$ an order Fra\"{i}ss\'e expansion of $K_{0}$ in $L$, $F=Flim(K)$ the Fra\"{i}ss\'e limit of $K$.
By Theorem 5.2$(ii)\Rightarrow(i)$ of \cite{KPT05}, if we denote $F_{0}=Flim(K_{0})$ then $F_{0}=F|L_{0}$.
Let $G_{0}=Aut(F_{0})$ and $G=Aut(F)$. Denote $<^{F}=<_{0}$, i.e.
$<_{0}$ is the linear order corresponding to the symbol $<$ in $F$,
and let $X_{K}=\overline{G_{0}<_{0}}$ ($X_{K}$ is called set of
\textit{$K$-admissible} linear orderings of $F$ in \cite{KPT05}). In \cite{KPT05}, two combinatorial properties for $K$ have considerable importance in order to compute universal minimal spaces. Those are called \emph{ordering property} and \emph{Ramsey property}: 

\begin{defn}

Let $\{<\}\subset L$ be a signature, $L_{0}=L\setminus\{<\}$, $K_{0}$ a Fra\"{i}ss\'e class in $L_{0}$, $K$ an order Fra\"{i}ss\'e expansion of $K_{0}$ in $L$, $F=Flim(K)$ the Fra\"{i}ss\'e limit of $K$. We say
that $K$ satisfies the \textbf{ordering property} (relative to $K_{0}$) if for every $A_{0}\in K_{0}$,
there is $B_{0}\in K_{0}$, such that for every linear ordering $\prec$
on $A_{0}$ and linear ordering $\prec'$ on $B_{0}$, if $A=\langle A_{0},\prec\rangle\in K$
and $B=\langle B_{0},\prec'\rangle\in K$ , then there is an embedding
$A\hookrightarrow B$.
\end{defn}

\begin{defn}
Let $\{<\}\subset L$ be a signature and $K$ be an order Fra\"{i}ss\'e class in $L$. We say that $K$ satisfies the \textbf{Ramsey property} if, for every positive $k \in \N$, every $A \in K$ and every $B  \in K$, there exists $C \in K$ such that for every $k$-coloring of the substructures of $C$ which are isomorphic to $A$, there is a substructure $\tilde B$ of $C$ which is isomorphic to $B$ and such that all substructures of $\tilde B$ which are isomorphic to $A$ receive the same color. 
\end{defn}

Those two properties are relevant because they capture dynamical properties of $X_{K}$. For example, Theorem 7.4 of \cite{KPT05} states that the minimality of $X_{K}$
is equivalent to $K$ having the ordering property, and Theorem 10.8 of
\cite{KPT05} states that $X_{K}$ being universal and minimal is equivalent to $K$ having the ordering and Ramsey properties.
Those results naturally led the authors of \cite{KPT05} to ask whether $X_{K}$ being universal is equivalent to $K$ having the Ramsey property. This question is precisely the reason for which the concept of relative extremely amenability was introduced. Recall that by Theorem 4.7 of \cite{KPT05}, the Ramsey property of
$K$ is equivalent to $G$ being extremely amenable. In \cite{NVT-},
it is shown that the universality of $X_{K}$ is equivalent to $(G_{0},G)$
being relatively extremely amenable. However, it is still unknown whether $(G_{0},G)$
being relatively extremely amenable is really weaker than $G$ being extremely amenable (see Section \ref{sub:Conjecture.} for more about this aspect). 

\begin{rem}

\label{ord}
The reason for which only \emph{order} expansions (i.e. $\{<\}\subset L$,$L_{0}=L\setminus\{<\}$, and $<$ is interpreted as a linear order) were considered in \cite{KPT05} is that, at the time where the article was written, expanding the signature by such a symbol was sufficient in order to obtain Ramsey property and ordering property in all known practical cases. However, we know now that there are some cases where expanding the language with more symbols is necessary (E.g. \emph{circular tournaments} and \emph{boron tree structures}, whose Ramsey-type properties have been respectiveley analyzed by Laflamme, Nguyen Van Th\'e and Sauer in \cite{LNS10}, and by Jasi\'nski in \cite{J-}). 
The description of the corresponding universal minimal spaces is very similar to what is obtained in \cite{KPT05} and will appear in a forthcoming paper. For the sake of clarity, we will only treat here the case of order expansions, which extends to the general case without difficulty.

\end{rem}

\subsection{The weak ordering property.}

\label{sub:weak-ordering-Property}

Theorem 10.8 of \cite{KPT05} states that $K$ has the ordering and
Ramsey properties if and only if $X_{K}$ is the universal minimal
space of $G_{0}$. The purpose of this section is to show that the combinatorial assumptions made on $K$ can actually be slightly weakened. We start with a generalization of the notion of transitivity mentioned
in subsection \ref{sub:Universal-spaces.}.
\begin{defn}
Let $G$ be a topological group and $X$ a $G$-space. $Y\subset X$
is said to be \textbf{transitive w.r.t}\textbf{\textit{ $X$}} if
and only if for any $y\in Y$, $\overline{Gy}=X$. \end{defn}
\begin{prop}
\label{thm:rea+transitive-->universal minimal}Let $G_{0}$ be a topological
group and let $T_{G_{0}}$ be $G_{0}$-universal. Let $x\in T_{G_{0}}$
and let $G=Stab{}_{G_{0}}(x)\subset G_{0}$. $T_{G_{0}}$ is minimal
if and only if $Fix_{T_{G_{0}}}(G)$ is transitive w.r.t $T_{G_{0}}$.\end{prop}
\begin{proof}
If $T_{G_{0}}$ is minimal then $T_{G_{0}}$ is transitive w.r.t itself
and trivially $Fix_{T_{G_{0}}}(G)\subset T_{G_{0}}$ is transitive
w.r.t $T_{G_{0}}$. To prove the inverse direction, let $M\subset T_{G_{0}}$
be a $G_{0}$-minimal space. By Proposition \ref{thm:rea=00003D00003DFix on Universal}(\ref{enu:Universal T with H subset Fix(t)}),
$(G_{0},G)$ is relatively extremely amenable and therefore there
exists $t_{0}\in M\cap Fix_{T_{G_{0}}}(G)$. As $Fix_{T_{G_{0}}}(G)$
is transitive w.r.t $T_{G_{0}}$, conclude $T_{G_{0}}=\overline{G_{0}t_{0}}\subset M$,
so $T_{G_{0}}=M$ is minimal.
\end{proof}

The previous proposition enables us to prove the following equivalence: 

\begin{thm}
\label{thm:Fix_transitivity-->minimality} $(G_{0},G)$ is
relatively extremely amenable and $Fix_{X_{K}}(G)$ is transitive
w.r.t $X_{K}$ if and only if $X_{K}$ is the universal minimal
space of $G_{0}$.\end{thm}
\begin{proof}
As indicated previously, the universality of $X_{K}$ is equivalent
to the fact that $(G_{0},G)$ is relatively extremely amenable. By
Proposition \ref{thm:rea+transitive-->universal minimal}, given that
$X_{K}$ is universal, the minimality of $X_{K}$ is equivalent to
the fact that $Fix_{X_{K}}(G)$ is transitive w.r.t $X_{K}$ .%
{}
\end{proof}

We are now going to show how to reformulate Theorem \ref{thm:Fix_transitivity-->minimality} in terms of combinatorics.

\begin{defn}

Let $\{<\}\subset L$ be a signature, $L_{0}=L\setminus\{<\}$, $K_{0}$ a Fra\"{i}ss\'e class in $L_{0}$, $K$ an order Fra\"{i}ss\'e expansion of $K_{0}$ in $L$. We say
that $(K_{0}, K)$ has the \textbf{relative Ramsey property} if for every positive $k \in \N$, every $A_{0} \in K_{0}$ and every $B  \in K$, there exists $C \in K_{0}$ such that for every $k$-coloring of the substructures of $C_{0}$ isomorphic to $A_{0}$, there is an embedding $\phi : B|L_{0} \hookrightarrow C_{0}$ such that for any two substructures $\tilde A, \tilde A'$ of $B_{0}$ isomorphic to $A_{0}$, $\phi(\tilde A)$ and $\phi (\tilde A')$ receive the same color whenever $\tilde A$ and $\tilde A'$ support isomorphic structures in $B$.
\end{defn}

In what follows, the relative Ramsey property will appear naturally because of the following fact (see \cite{NVT-}): 

\begin{claim}
$(G_{0},G)$ is relatively extremely amenable iff $(K_{0},K)$ has the relative Ramsey property. 
\end{claim}

We will also need the following variant of the notion of ordering property: 

\begin{defn}
\label{def:weak ordering property}
Let $\{<\}\subset L$ be a signature, $L_{0}=L\setminus\{<\}$, $K_{0}$ a Fra\"{i}ss\'e class in $L_{0}$, $K$ an order Fra\"{i}ss\'e expansion of $K_{0}$ in $L$. We say that $K$ satisfies the \textbf{weak
ordering property} relative to $K_{0}$ if for every $A_{0}\in K_{0}$, there is $B_{0}\in K_{0}$,
such that for every linear ordering $\prec$ on $A_{0}$ with $A=\langle A,\prec\rangle\in K$
and linear ordering $\prec'\in Fix_{X_{K}}(G)$ we have $A\hookrightarrow\langle B_{0},\prec'|B_{0}\rangle$.
\end{defn}
The following claim appears in the proof of Theorem 7.4 of \cite{KPT05}:
\begin{claim}
\label{claim}Let $<$ be a linear ordering on $F_{0}$. Then $<_{0}\in\overline{G_{0}<}$
if and only if for every $A\in K$ there is a finite substructure
$C_{0}$ of $F_{0}$ such that $C=\langle C_{0},<|C_{0}\rangle\cong A$. \end{claim}
\begin{prop}
\label{thm:weak ordering->ordering property}Assume $K$ satisfies the weak ordering
property relative to $K_{0}$, and that $(K_{0}, K)$ has the relative Ramsey property. Then $K$ satisfies the ordering property.\end{prop}
\begin{proof}
Again, the universality of $X_{K}$ is equivalent
to the fact that $(G_{0},G)$ is relatively extremely amenable, which is in turn equivalent to $(K_{0}, K)$ having the relative Ramsey property. By
Theorem 7.4 of \cite{KPT05} the minimality of $X_{K}$ is equivalent
to the ordering property of $K$ (relative to $K_{0}$). By Proposition \ref{thm:rea+transitive-->universal minimal}
in order to establish $X_{K}$ is minimal, it is enough to show that
$Fix_{X_{K}}(G)$ is transitive w.r.t $X_{K}$. Let $<\in Fix_{X_{K}}(G)$.
It is enough to show $<_{0}\in\overline{G<}$. Fix $A\in K.$ As $K$
satisfies the weak ordering property, there is $B_{0}$ as in Definition
\ref{def:weak ordering property} such that $A\hookrightarrow\langle B_{0},<|B_{0}\rangle$.
Using the same argument as in the proof of Theorem 7.4 of \cite{KPT05},
we notice that there is a substructure $C$ of $B$ isomorphic to $A$.
Denote $C_{0}=C|L_{0}$ and notice $C=\langle C_{0},<|C_{0}\rangle\cong A$
. We now use Claim \ref{claim}.
\end{proof}

\begin{thm}
$K$ has the weak ordering property and $(K_{0}, K)$ has the relative Ramsey property if and only if $X_{K}$ is the
universal minimal space of $G_{0}$. \end{thm}
\begin{proof}
By Theorem 10.8 of \cite{KPT05}, if $X_{K}$ is the universal minimal
space of $G_{0}$ then $K$ satisfies the ordering property, a fortiori,
$K$ satisfies the weak ordering property. In addition $K$ satisfies
the Ramsey property which implies $(K_{0}, K)$ has the relative Ramsey property. The reverse direction follows from Proposition \ref{thm:weak ordering->ordering property}.
\end{proof}

\subsection{A question.}

\label{sub:Conjecture.}

We mentioned previously that the concept of relative extreme amenability was introduced in order to know whether $X_{K}$ being universal is equivalent to $K$ having the Ramsey property. By Theorem 4.7 of \cite{KPT05}, the Ramsey property of
$K$ is equivalent to $G$ being extremely amenable. We still do not know the answer to the following question from \cite{KPT05}: 

\begin{question}\label{Q1} Let $\{<\}\subset L$ be a signature, $L_{0}=L\setminus\{<\}$, $K_{0}$ a Fra\"{i}ss\'e class in $L_{0}$, $K$ an order Fra\"{i}ss\'e expansion of $K_{0}$ in $L$. Does universality for $X_{K}$ imply that $G$ is extremely amenable (equivalently, that $K$ has the Ramsey property)?\end{question}

Moreover, in view of the notions we introduced previously, we ask: 

\begin{question}
Assume the previous question has a negative answer. Does there exist an extremely amenable interpolant for the pair $(G_{0},G)$? 
\end{question}

As a final comment, and in view of Remark \ref{ord}, it should be mentioned that Question \ref{Q1} has a negative answer when $K$ is not an order expansion of $K_{0}$, see \cite{NVT-}.

\bibliographystyle{alpha} \bibliographystyle{alpha} \bibliographystyle{alpha}
\bibliographystyle{alpha}
\bibliography{bib2.bib}

\def\cprime{$'$}
\begin{thebibliography}{LNVTS10}

\bibitem[BP11]{BP-}
Manuel Bodirsky and Michael Pinsker.
\newblock Reducts of {R}amsey structures.
\newblock to appear in AMS Contemporary Mathematics, 2011.

\bibitem[Cam76]{Cam76}
Peter~J. Cameron.
\newblock Transitivity of permutation groups on unordered sets.
\newblock {\em Math. Z.}, 148(2):127--139, 1976.

\bibitem[dV93]{dV93}
Jan de~Vries.
\newblock {\em Elements of topological dynamics}, volume 257 of {\em
  Mathematics and its Applications}.
\newblock Kluwer Academic Publishers Group, Dordrecht, 1993.

\bibitem[Fra00]{F00}
Roland Fra{\"{\i}}ss{\'e}.
\newblock {\em Theory of relations}, volume 145 of {\em Studies in Logic and
  the Foundations of Mathematics}.
\newblock North-Holland Publishing Co., Amsterdam, revised edition, 2000.
\newblock With an appendix by Norbert Sauer.

\bibitem[GL13]{GL11}
Yonatan Gutman and Hanfeng Li.
\newblock A new short proof of the uniqueness of the universal minimal space.
\newblock {\em Proc. Amer. Math. Soc.}, 141(1):265--267, 2013.

\bibitem[Hod93]{H93}
Wilfrid Hodges.
\newblock {\em Model theory}, volume~42 of {\em Encyclopedia of Mathematics and
  its Applications}.
\newblock Cambridge University Press, Cambridge, 1993.

\bibitem[Jas13]{J-}
Jakub Jasi\'nski.
\newblock Ramsey degrees of boron tree structures.
\newblock {\em Combinatorica}, 33(1):23--44, 2013.

\bibitem[Kec95]{K95}
Alexander~S. Kechris.
\newblock {\em Classical descriptive set theory}, volume 156 of {\em Graduate
  Texts in Mathematics}.
\newblock Springer-Verlag, New York, 1995.

\bibitem[KPT05]{KPT05}
Alexander~S. Kechris, Vladimir~G. Pestov, and Stevo Todorcevic.
\newblock Fra\"\i ss\'e limits, {R}amsey theory, and topological dynamics of
  automorphism groups.
\newblock {\em Geom. Funct. Anal.}, 15(1):106--189, 2005.

\bibitem[LNVTS10]{LNS10}
Claude Laflamme, Lionel Nguyen Van~Th{\'e}, and Norbert~W. Sauer.
\newblock Partition properties of the dense local order and a colored version
  of {M}illiken's theorem.
\newblock {\em Combinatorica}, 30(1):83--104, 2010.

\bibitem[Mit66]{M66}
Theodore Mitchell.
\newblock Fixed points and multiplicative left invariant means.
\newblock {\em Trans. Amer. Math. Soc.}, 122:195--202, 1966.

\bibitem[NVT13]{NVT-}
Lionel Nguyen Van~Th\'e.
\newblock Universal flows of closed subgroups of ${S}_{\infty}$ and relative
  extreme amenability.
\newblock In {\em {Proceedings of the Fields Institute Thematic Program on
  Asymptotic Geometric Analysis}}, volume~68 of {\em Fields Institute
  Communications}, pages 229--245. Springer, 2013.

\bibitem[Pes98]{P98}
Vladimir~G. Pestov.
\newblock On free actions, minimal flows, and a problem by {E}llis.
\newblock {\em Trans. Amer. Math. Soc.}, 350(10):4149--4165, 1998.

\bibitem[Pes02]{P02}
Vladimir~G. Pestov.
\newblock Ramsey-{M}ilman phenomenon, {U}rysohn metric spaces, and extremely
  amenable groups.
\newblock {\em Israel J. Math.}, 127:317--357, 2002.

\bibitem[Usp90]{U90}
Vladimir~V. Uspenskij.
\newblock On the group of isometries of the {U}rysohn universal metric space.
\newblock {\em Comment. Math. Univ. Carolin.}, 31(1):181--182, 1990.

\bibitem[Usp02]{Usp01}
Vladimir~V. Uspenskij.
\newblock Compactifications of topological groups.
\newblock In {\em Proceedings of the {N}inth {P}rague {T}opological {S}ymposium
  (2001)}, pages 331--346. Topol. Atlas, North Bay, ON, 2002.

\bibitem[Vee77]{V77}
William~A. Veech.
\newblock Topological dynamics.
\newblock {\em Bull. Amer. Math. Soc.}, 83(5):775--830, 1977.

\end{thebibliography}

\address{Yonatan Gutman, Laboratoire d'Analyse et de Math\'ematiques
Appliqu\'ees,
Universit\'e de Marne-la-Vall\'ee, 5 Boulevard Descartes, Cit\'e Descartes
- Champs-sur-Marne, 77454 Marne-la-Vall\'ee cedex 2, France \&
Institute of Mathematics, Polish Academy of Sciences,
ul. \'{S}niadeckich~8, 00-956 Warszawa, Poland.}

\email{y.gutman@impan.pl, yonatan.gutman@univ-mlv.fr}

\address{Lionel Nguyen Van Th\'e, Aix Marseille Universit\'e, CNRS, LATP, UMR 7353, 
13453 Marseille France}

\email{lionel@latp.univ-mrs.fr}

\end{document}